\documentclass[12pt]{article}
\usepackage{amsmath}
\usepackage[dvipsnames,usenames]{color}
\usepackage{amsfonts}
\usepackage{mathrsfs}
\usepackage{amssymb}
\usepackage{amsthm}
\def\sm{\setminus}
\def\<{\langle}
\def\>{\rangle}
\numberwithin{equation}{section}
\def\<{\langle}
\def\>{\rangle}

\def\ra{\rightarrow}

\def \sm{\setminus}
\def\-{\overline}

\def\-{\overline}

\def\sm{\setminus}

\def\ra{\rightarrow}

\newtheorem{theorem}{Theorem}[section]
\newtheorem{lemma}[theorem]{Lemma}
\newtheorem{corollary}[theorem]{Corollary}
\newtheorem{proposition}[theorem]{Proposition}

\newtheorem{conjecture}[theorem]{Conjecture}
\newtheorem{example}[theorem]{Example}

\date{\ }

\begin{document}
\title{\bf Bergman-Einstein metric on  a Stein space with a strongly pseudoconvex boundary}

\author{{Xiaojun  Huang}\footnote{Supported by NSF grant DMS-2000050}
\qquad
Xiaoshan Li\footnote{Supported by  NSFC grant  No. 11871380}}
\maketitle

\begin{abstract}\vskip 3mm\footnotesize
\noindent Let $\Omega$ be a Stein space with a compact smooth
strongly pseudoconvex boundary. We prove that the boundary is
spherical if its Bergman metric over $\hbox{Reg}(\Omega)$ is
K\"ahler-Einstein. \vskip 4.5mm



\end{abstract}
\section{Introduction}
For any bounded  domain in $D\subset \mathbb C^n$,  its Bergman
metric is a canonical biholomorphically invariant  K\"ahler metric
over $D$. Cheng-Yau \cite{CY80} proved that there exists a complete
K\"ahler-Einstein metric on a bounded pseudoconvex domain in
$\mathbb C^n$ with a $C^2$-smooth boundary.   A well-known open
question initiated from the work of Cheng-Yau \cite{CY80}  asks when
the Bergman metric on a smoothly bounded  domain coincides with
 its Cheng-Yau K\"ahler-Einstein metric.
Cheng conjectured in \cite{C79}  that the Bergman metric of a
smoothly bounded strongly pseudoconvex domain is K\"ahler-Einstein
if and only if the domain is biholomorphic to the ball. This
conjecture was solved by Fu-Wong \cite{FW97} and Nemirovski-Shafikov
\cite{NS06} in the case of complex dimension two and was  verified
in a recent paper of Huang-Xiao \cite{HX16} for any dimensions. Recently, Ebenfelt-Xiao-Xu \cite{EXX20}  introduced a new characterization of the two-dimensional unit ball $\mathbb B^2$,
more generally, two-dimensional finite ball quotients $\mathbb B^2/\Gamma$ in terms of algebracity of the Bergman kernel.
There have been also other related studies on versions of the
Cheng's conjecture in terms of metrics defined by other important
canonical potential functions as in the work of Li \cite{L1, L2,
L3}.

On a complex space $\Omega$ with possible singularities, Kobayashi
\cite{Kob} defined the Bergman kernel  form on its smooth part ${\rm
Reg}(\Omega)$ which is naturally  identified with the Bergman kernel
function in the domain case. The Kobayashi Bergman kernel form can
be similarly used to define a K\"ahler form on ${\rm Reg}(\Omega)$
under certain geometric conditions  on $\Omega$, which are always
the case when $\Omega$ is a Stein space with a compact smooth
strongly pseudoconvex boundary. In this paper, we address the
generalized Cheng question of understanding the geometric
implication when the Bergman metric of a Stein space with a compact
strongly pseudoconvex boundary has the Einstein property.

To state our main theorem, we first introduce a few  notations. Let
$\Omega$ be a stein space of dimension $n$  with possibly isolated
singularity and write ${\rm Reg}(\Omega)$  for its regular part.
Write $\Lambda^n({\rm Reg}(\Omega))$  for the space of the
holomophic $(n, 0)$-forms on ${\rm Reg}(\Omega)$ and define the
Bergman space of $\Omega$ as follows:
$$A^2(\Omega):=\{f\in \Lambda^n({\rm
Reg}(\Omega)):(-1)^{\frac{n^2}{2}}\int_{{\rm Reg}(\Omega)} f\wedge
\overline f<\infty\}.$$  Then $A^2(\Omega)$ is a Hilbert sapce with
the inner product:
$$(f, g)=(-1)^{\frac{n^2}{2}}\int_{{\rm Reg}(\Omega)}f\wedge\overline g, ~\text{for all}~f,
 g\in \Lambda^n({\rm Reg}(\Omega)).$$ We assume that $A^2(\Omega)\not =\{0\}$.
  Let $\{f_j\}_1^N$ be an orthonormal basis of $A^2(\Omega)$ and
  define the Bergman kernel to be $K_\Omega=\sum_{j=1}^N f_j\wedge\overline f_j$. Here, $N$ is either
  a natural number or $\infty$.
   In a local holomorphic coordinate chart $(U, z)$ on ${\rm Reg}(\Omega)$, we have
$$K_\Omega=k_\Omega(z, \overline z) dz_1\wedge\cdots\wedge dz_n\wedge d\overline z_1\wedge\cdots\wedge d\overline z_n~\text{in}~ U.$$
Assume further that $K_\Omega$ is nowhere zero on ${\rm
Reg}(\Omega)$. We define a Hermitian $(1, 1)$-form on ${\rm
Reg}(\Omega)$ by $\omega^B_\Omega=i\partial\overline\partial \log
k_\Omega(z, \overline z).$ We call $\omega^B_\Omega$ the Bergman
metric  on $\Omega$ if it indeed induces a positive definite metric
on ${\rm Reg}(\Omega)$.

Notice that if $\Omega$ is a Stein space with a compact smooth
strongly pseudoconvex boundary then $\overline{\Omega}$ can be
compactly embedded into a closed  Stein subspace of a certain
complex Euclidean space. Then $A^2(\Omega)$ is of infinite dimension
and it indeed defines a Bergman metric on ${\rm Reg}(\Omega)$.

Our main purpose  of this paper is to generalize  results obtained
in \cite{FW97} and \cite{HX16} to Stein spaces:
\begin{theorem}\label{mthm-4-18}
Let $\Omega$ be a Stein space with  a compact smooth strongly
pseudoconvex boundary. If its Bergman metric $\omega_\Omega^B$ on
${\rm Reg}(\Omega)$ is K\"ahler-Einstein then $\partial\Omega$ is
spherical.
\end{theorem}

\section{Proof of Theorem \ref{mthm-4-18}}\label{4-19-a4}

In this section, we start with a strongly pseudoconvex complex manifold $M$ with  compact strongly pseudoconvex boundary. We denote by $E$ the exceptional set in $M$ in the sense of Grauert \cite{G62}, that is, there exists a blowing down map $\pi: M\rightarrow\Omega$ from $M$ to a stein space $\Omega$ with isolated singularities such that $\pi^{-1}({\rm Sing}(\Omega))=E$ and $\pi: M\setminus E\rightarrow \Omega\setminus {\rm Sing}(\Omega)$ is a biholomorphic map. Here, we denote by ${\rm Sing}(\Omega)$ the set of singularities in $\Omega$ and define ${\rm Reg}(\Omega):=\Omega\setminus{\rm Sing}(\Omega)$. Since the boundary of $M$ is strongly pseudoconvex then by a Theorem of
Oshawa \cite{Oh84} and Hill-Nacinovich \cite[Theorem 3.1]{HN05} there exists a larger complex manifold $M'$ which include $\overline M$ and $M$ as an open subset.

Let $\Omega^{n, 0}(\overline M)$ be the space of smooth $(n, 0)$-forms on $M$ which are smooth up to the boundary. Let $\Omega^{n, 0}_c(M)$ be the subspace of $\Omega^{n, 0}(\overline M)$ with elements  having compact support in $M$. We define the $L^2$ inner product on $\Omega_c^{n, 0}(M)$ as following
$$(f, g)=(-1)^{\frac{n^2}{2}}\int_M f\wedge \overline g~\text{for all}~f, g\in\Omega_c^{n, 0}(M).$$
Let $L^2_{(n, 0)}(M)$ be the completion of $\Omega_c^{n, 0}(M)$
under the above inner product. We denote by $H_s(M), s\in\mathbb R$
the Sobolev space of order $s$ on $M$ (see \cite[Appendix]{FK72}).
Write $\Lambda^n(M)$ for the space of the holomorphic $n$-forms on $M$ and we
 define the Bergman space of $M$ to be
 $$A^2(M)=\left\{f\in \Lambda^n(M): (-1)^{\frac{n^2}{2}}\int_M f\wedge\overline f
 <\infty\right\}.$$ Then $A^2(M)$ is a closed subspace of $L^2_{(n,
 0)}(M)$.

Let $P: L^2_{(n, 0)}(M)\rightarrow A^2(M)$ be the orthogonal
projection which we call the Bergman projection of $M$. The
reproducing kernel of the Bergman projection is denoted by $K_{M}(z,
w)$. Let $\{f_j\}_{j=1}^\infty$ be an orthnormal basis of $A^2(M)$.
Let $pr_1: M\times M\rightarrow M$ and $pr_2: M\times M\rightarrow
M$ be the natural projection from the product space. Then the
reproducing kernel of the Bergman projection $P$  is a $2n$-form on
$M\times M$ which  can be written as
$$K_M(z, \overline w)=\sum_{j=1}^\infty pr_1^\ast f_j\wedge pr_2^\ast \overline {f_j}=\sum_{j=1}^\infty f_j(z)\wedge \overline {f_j(w)}, \forall (z, w)\in M\times M.$$
Here, $f_j(z)$ and $f_j(w)$ are considered as a $(n, 0)$-forms at $(z, w)$ for each $j$. Then
$K_M(z, \overline z)$ can be considered as a $2n$-form on $M$ which is called Bergman
 kernel on $M$. Both $K_M(z, \overline w)$ and the Bergman kernel $K_M(z, \overline z)$ are
  independent of the choice of the orthonormal basis of $A^2(M)$. In a local coordinate
  chart $(U, z)$ of $M$ with $z=(z_1, \ldots, z_n)$ we have
\begin{equation}\label{20/5/10/a1}
K_M(z, \overline z)=k_M(z, \overline z)dz_1\wedge\cdots\wedge
dz_n\wedge d\overline{z_1}\wedge\cdots\wedge d\overline{z_n},
\end{equation} where
$k_M(z, \overline z)=\sum_{j=1}^\infty |\hat f_j(z)|^2$ with
$f_j=\hat f_j(z)dz_1\wedge\cdots\wedge dz_n$. Then
$\omega^B_M=\partial\overline\partial \log k_M$  is a well defined
Hermitian $(1, 1)$-form on $M$ where $K_M$ is nonzero. We call
$\omega^B_M$ the Bergman metric over the subset where it is
positive definite.

Since the Bergman metric over ${\rm Reg}(\Omega)$ is well defined, thus $\omega^B_M$ is a well defined Bergman metric on $M\setminus E$.  Write $g^M_{\alpha\overline\beta}=\frac{\partial^2\log k_M}{\partial z_\alpha\partial\overline z_\beta}$ and define $G_M(z):=\det(g^M_{\alpha\overline \beta})$. Then the Ricci
tensor of the Bergman metric on $M\setminus E$ is given by
$$R^M_{\alpha\overline\beta}(z)=-\frac{\partial ^2\log G_M(z)}{\partial z_\alpha\partial
\overline z_\beta}.$$ The Bergman metric on $M\setminus E$ is called K\"ahler-Einstein
when $R^M_{\alpha\overline\beta}=c g^M_{\alpha\overline \beta}$ for some constant $c$.
 It is well-known that the constant
$c$ is necessary negative (as we will also see later).  Since
$\omega^B_M=\pi^\ast \omega_\Omega^B$ over $M\setminus E$,
 thus $\omega^B_M$ is K\"ahler-Einstein over $M\setminus E$ if and only if $\omega_\Omega^B$
 is K\"ahler-Einstein over ${\rm Reg}(\Omega)$.

Now, an equivalent version of Theorem \ref{mthm-4-18} is as follows:
\begin{theorem}\label{mthm1}
Let $M$ be a complex manifold with a compact smoothly stronlgy pseudoconvex pseudoconvex boundary. If the Bergman metric on $M\setminus E$ is Kahler-Einstein, then $\partial M$ is spherical.
\end{theorem}

With Theorem \ref{mthm1} at our disposal  and by a similar argument
as in the \cite{NS06} and \cite{HX16}, we have the following:

\begin{corollary}\label{mthm2}
Let $M$ be a Stein manifold with a compact smooth stronlgy
pseudoconvex pseudoconvex boundary. If the Bergman metric on $M$ is
Kahler-Einstein, then $M$ is biholomorphic to the ball.
\end{corollary}

\section{Localization of Bergman kernel forms }\label{localization}
Assume now that $M$ is a complex manifold with a compact smooth
strongly pseudo-convex boundary.
 Fix $w_0\in M$. Then $K_M(z, w_0)$ is a
holomorphic $(n, 0)$-form with respect to $z$ and is
$L^2$-integrable. Let $w=(w_1, \cdots, w_n)$ be coordinates in a
neighborhood of $w_0$. We explain the meaning of $L^2$-integrable of
$K_M(z, w_0)$. Write $dw=dw_1\wedge\cdots\wedge dw_n$ and
$d\overline w=d\overline w_1\wedge\cdots d\overline w_n$. Then write
$$K_M(z, w_0)=\tilde k_M(z, w_0) \wedge d\overline w|_{w_0}.$$ Here,
$\tilde k_M(z, w_0)$ is a $(n, 0)$-form on $M$. We say $K_M(z, w_0)$
is $L^2$-integralbe with respect to $z$ if
$$(-1)^{\frac{n^2}{2}}\int_M \tilde k_M(z, w_0)\wedge
\overline{\tilde k_M(z, w_0)}<\infty.$$ The $L^2$-integrability of
$K(z, w_0)$ does not depend on the choice of coordinates $w$.

For any $p\in\partial M$, there exists a coordinate chart $(U, z)$ of $M'$ centered at $p$. Take a smooth strongly pseudocovnex domain $D\subset M\cap U$ such that
\begin{equation}\label{2019-09-30-a1}
D\cap B(p, 2\delta)=M\cap B(p, 2\delta)
\end{equation}
where $B(p, 2\delta)=\{q\in U: |z(q)|<2\delta\}$ where
$|z|=\sqrt{|z_1|^2+\cdots+|z_n|^2}$. Here, $\delta$ is sufficiently
small. We then have the following localization result for which
there is no need to assume that the Bergman metric of $M$ is
K\"ahler-Einstein.
\begin{proposition}\label{localization of Bergman kernel}
For $p\in\partial M$, let $D\subset M$ be a strongly pseudoconvex domain  satisfying   (\ref{2019-09-30-a1}). Let $k_M(z, \overline z), k_D(z, \overline z)$ be given as in (\ref{20/5/10/a1}). Then
\begin{equation}
k_M(z, \overline z)=k_D(z, \overline z)+\varphi(z),
\end{equation}
where $\varphi(z)\in C^{\infty}( B(p, \delta)\cap \overline
M)$.
\end{proposition}
\begin{proof}
We use the  Fefferman \cite{Fe74} localization method
     developed in the domain case.  For clarity,  we proceed  in  two steps.

{\bf Step 1.} Let $(U, w)$ be a coordinate chart centered at $p$
where $w=(w_1, \cdots, w_n)$ are holomorphic coordinates. Write
$d\overline w|_w=d\overline w_1\wedge\cdots\wedge d\overline w_n|_w,
\forall w\in U$.
     We fix $w\in B(p, r)\cap M$ and set
    $$f_w(z)=K_M(z, \overline w)-K_D(z, \overline w)\chi_D(z), z\in M,$$
    where $\chi_D$ is the characteristic function of $D$.
    Write $f_w(z)=\tilde f_w(z)\wedge d\overline w|_w$ and $\tilde g_w(z)=\overline\partial
    \tilde f_w$ where $\tilde f_w(z)$ is a $L^2$-integrable $(n, 0)$-form on $M$,
     $\tilde f_w\perp A^2(M)$ and $\tilde g_w$ is a $(n, 1)$-form in $H_{-1}(M)$ with
     ${\rm supp}~\tilde g_w\subset \partial D\setminus\partial M$.
     By the smoothing property,
      there is a sequence of $(n, 0)$-form $\{\tilde f_w^\varepsilon\}$ on $M$ which
      are smooth up to $\overline M$ such that
$\tilde f_w^\varepsilon\rightarrow \tilde f_w~\text{in the
}~L^2~\text{space}.$ Set $\tilde g_w^\varepsilon=\overline\partial
\tilde f_w^\varepsilon$.
 Since ${\rm supp}~\tilde g_w\subset \overline{ \partial D\setminus \partial M}$,
 we can assume that
  ${\rm supp}~\tilde g_w^\varepsilon$ is contained in a $\varepsilon$-neighborhood of $\partial D\setminus \partial M$. Moreover,

\begin{equation}\label{10-3-a1}
\tilde f_w^\varepsilon\rightarrow \tilde f_w~\text{in}~L_{(n, 0)}^2(M), \tilde g_w^\varepsilon\rightarrow \tilde g_w~\text{in}~H_{-1}(M).
\end{equation}

Fix a Hermitian metric $g$ on $M'$.  For $0\leq q\leq n$, let $L^2_{(n, q)}(M)$ be
 the space of $L^2$-integrable $(n, q)$-forms with respect to $g$. When $q=0$,
 this definition of $L^2_{(n, 0)}(M)$ is the same as defined  in Section \ref{4-19-a4}.
  We denote by $N^{(q)}$  the $\overline\partial$-Neumann operator with respect
   to $\Box^{(q)}$.  For convenience, we denote $N^{(q)}$  by $N$ when it dose not cause any confusing. Since $M$ is strongly pseudoconvex, then by the local regularity of $N$ \cite{FK72} we have
\begin{equation}\label{4-18-a1}
\|\xi N \tilde g_w^\varepsilon\|_{s+1}\leq C_s( \|\xi_1 \tilde g_w^\varepsilon\|_s+\|\tilde g_w^\varepsilon\|_{-1}), \forall s\geq 0,
\end{equation}
with $\{C_s\}$ constants independent of $w$.
Here, $\xi(z), \xi_1(z)\in C_0^\infty(B(p, \frac{3}{2}\delta))$ and $\xi_1|_{{\rm supp}\xi}\equiv1$, $\xi|_{B(p, \delta)}\equiv1$. Since $B(p, 2\delta)\cap \partial D\setminus \partial M= \emptyset$, then $\xi_1 \tilde g_w^\varepsilon \equiv0$ when $\varepsilon$ is sufficiently small. Thus,
\begin{equation}\label{10-3-a2}
\|\xi N \tilde g_w^\varepsilon\|_{s+1}\leq C_s\|\tilde g_w^\varepsilon\|_{-1}.
\end{equation}
By (\ref{10-3-a1}) and (\ref{10-3-a2}), $\{\xi N\tilde
g_w^\varepsilon\}$ is a  Cauchy  sequence in $H_{s+1}(M)$ for any
$s\geq 0$. Assume that $\xi N \tilde g_w^\varepsilon\rightarrow h$
in $H_s(M)$ for any $s\geq 0$. Then $h\in C^\infty(\overline M)$. On
the other hand, $\tilde f_w^\varepsilon-P \tilde
f_w^\varepsilon=\overline\partial^\ast N \tilde g_w^\varepsilon$
where $P: L^2_{(n, 0)}(M)\rightarrow A^2(M)$ is the Bergman
projection. Then
\begin{equation}\label{10-3-a5}
\xi(\tilde f_w^\varepsilon-P\tilde f_w^\varepsilon)=\xi\overline\partial^\ast N \tilde g_w^\varepsilon=\overline\partial^\ast (\xi N \tilde g_w^\varepsilon)-[\xi, \overline\partial^\ast] (\xi_1N\tilde g_w^\varepsilon).
\end{equation}

By (\ref{10-3-a2}), we have
\begin{equation}\label{10-3-a4}
\|\xi(\tilde f_w^\varepsilon-P \tilde f_w^\varepsilon)\|_s\leq
C_s\|\tilde g_w^\varepsilon\|_{-1}.
\end{equation}

We claim that $\{\|\tilde g_w\|_{-1}\}$ has uniform bound with respect
 to $w\in B(p, \delta)\cap M$. We next give a proof of this Claim as
 follows:

  Choose a real function $\rho\in C^\infty(M')$ such
  that $\rho\equiv 1$ in a $2\sigma$-neighborhood
    of $\partial D\setminus \partial M$ denoted by $V_{\sigma}$ in $M'$. Write $K_D(z, w)=\tilde K_D(z, w)\wedge d\overline w|_w$ for all $w\in M\cap B(p, \delta)$. Since ${\rm  supp}~\tilde g_w\subset\partial D\setminus\partial M$, then $\forall \varphi=\sum_{j=1}^n\varphi_j dz_1\wedge\cdots\wedge dz_n\wedge d\overline z_j\in \Omega^{(n, 1)}_c(M)$ we have $(\tilde g_w,  \varphi)=(\tilde g_w, \rho\varphi)$ and
\begin{equation}\label{10-4-a1}
\begin{split}
(\tilde g_w, \rho\varphi)&=(\overline\partial \tilde f_w, \rho\varphi)=(\overline\partial (\tilde K_D(z, \overline w)\chi_D(z)), \rho\varphi)\\
&=(\tilde K_D(z, \overline w)\chi_D(z), \overline\partial^\ast(\rho\varphi))=\int_D \tilde K_D(z, \overline w)\wedge \overline{\overline\partial^\ast (\rho\varphi)}\\
&=\int_{V_{2\sigma}} k_D(z, w)dz_1\wedge\cdots\wedge dz_n\wedge\overline{\overline\partial^\ast(\rho\varphi)},
\end{split}
\end{equation}
where $\tilde K_D(z, w)=k_D(z, w)dz_1\wedge\cdots\wedge dz_n$. Since $d(V_{2\sigma}, B(p, \delta))>0$ when $\sigma, \delta$ are sufficinetly small then by a result of Kerzman \cite[Theorem 2]{Ke72} we have
\begin{equation}\label{10-6-a1}
\sup_{z\in V_{\sigma}}|k_D(z, w)|\leq C, \forall w\in M\cap B(p, \delta)
\end{equation}
where $C$ is a constant independent of $w$.
Then from (\ref{10-4-a1}) and (\ref{10-6-a1}) we have
\begin{equation}
|(g_w, \varphi)|\leq C_1\|\varphi\|_1, \forall w\in B(p, \delta)\cap M,
\end{equation}
where the constant $C_1$ does not depend on $w\in B(p, \delta)\cap M$. Thus, we get the conclusion of the Claim.

On the other hand, $P\tilde f_w^\varepsilon\rightarrow 0$ in
$L^2(M)$ because $\tilde f_w\perp A^2(M)$. By (\ref{10-3-a5}) and
the Rellich lemma, we have $\xi(\tilde f_w^\varepsilon-P\tilde
f_w^\varepsilon)\rightarrow h_s$ in $H_s(M)\  \forall s\geq 0$ for a
certain $h_s$. Then by (\ref{10-3-a1}) we have $h_s=\xi\tilde f_w$.
Thus, from the above Claim and by taking the limit in
(\ref{10-3-a4}), we have
\begin{equation}
\|\xi \tilde f_w\|_{s}\leq \tilde C_s.
\end{equation}
Here, the constant $\tilde C_s$ does not depend on $w\in B(p, r)\cap
M$.

\medskip
{\bf Step 2.} Write $f_w(z)=\tilde f_w(z)dw|_w$ and
$\tilde g_w=\overline\partial\tilde  f_w$.
Then $D_w^\alpha \tilde g_w=\overline\partial D_w^\alpha\tilde f_w$
for any multi-index $\alpha=(\alpha_1, \ldots, \alpha_n)$. Here, $\overline\partial$ is
 defined with respect to the $z$-direction.
 We still have $D_w^\alpha \tilde f_w\perp A^2(M)$ for any $w\in M\cap B(p, \delta)$.
 Then by a similar argument in { Step 1}, we have
\begin{equation}\label{10-3-a6}
\|\xi D_w^\alpha \tilde f_w\|_s\leq \tilde C_s.
\end{equation}
Here, constants $\tilde C_s$ do not depend on $w\in M\cap B(p,
\delta)$. Then by Sobolev embedding theorem, we have that
\begin{equation}\label{10-3-a7}
|\xi D_z^\alpha D_w^\beta\tilde f_w(z)|\leq C_{\alpha, \beta}, \forall \alpha, \beta, \forall z\in M, w\in M\cap B(p, \delta),
\end{equation}
where $C_{\alpha, \beta}$ are constants. Since $\xi|_{B(p, \delta)}\equiv1$,
 thus (\ref{10-3-a7}) implies that $\tilde f_w(z)$ is smooth up to
 $B(p, \delta)\cap \overline M\times B(p, \delta)\cap \overline M$. Thus,
 we get the conclusion of the proposition  if we take $z=w \in B(p, \delta)\cap \overline M$.
\end{proof}

Let $B_M(z)=G_M(z)/k_M(z, z)$. Then $B_M(z)$ is a globally-defined
smooth function on $M$ although $G_M(z)$ and $k_M(z, z)$ are only
locally given. The following lemma is a generalization of a result
of Diederich  \cite[Theorem 2]{Di70}:

\begin{lemma}\label{2019-09-30-lem1}
    $B_M(z)\rightarrow \frac{(n+1)^n\pi^n}{n!}$ as $z\rightarrow \partial M$.
\end{lemma}

\begin{proof}
By Lemma \ref{localization of Bergman kernel}, for any $p\in\partial M$ there exists a strongly pseudocovnex domain $D\subset M$ which satisfies \ref{2019-09-30-a1} such that
\begin{equation}\label{2019-09-30-a2}
k_M(z, \overline z)=k_D(z, \overline z)+\varphi(z)
\end{equation}
where $\varphi(z)\in C^\infty(B(p, \delta)\cap\overline M)$.
Then
\begin{equation}
\log {k_M(z, \overline z)}=\log k_D(z, \overline z)+\log \left(1+\frac{\varphi(z)}{k_D(z, \overline z)}\right), z\in D\cap B(p, \delta).
\end{equation}
Thus,
\begin{equation}\label{2019-09-30-a3}
g^M_{\alpha\overline\beta}=g^D_{\alpha\overline\beta}+\frac{\partial^2}{\partial z_\alpha\partial \overline z_\beta}\log \left(1+\frac{\varphi(z)}{k_D(z, \overline z)}\right).
\end{equation}
Since $D$ can be seen as a strongly pseudoconvex domain in $\mathbb C^n$ with
 smooth boundary, then by Fefferman's asymptotic expansion of Bergman kernels, we have
\begin{equation}\label{Bergman asymptotic expansion}
k_D(z, \overline z)=\frac{\Phi(z)}{r^{n+1}(z)}+\Psi(z)\log {r(z)}, z\in D.
\end{equation}
where $r$ is a Fefferman defining function for $D$ and $\Phi, \Psi\in C^\infty(\overline D)$ and $\Phi(z)\neq 0$ for all $z\in\partial D$.  Then
\begin{equation}\label{4-26-a1}
\log{(1+\frac{\varphi}{k_D(z, \overline z)})}=\log{\left(1+\frac{\varphi(z)r^{n+1}}{\Phi+\Psi r^{n+1}\log r}\right)}=\log{(1+f r^{n+1})}
\end{equation}
where $f=\frac{\varphi(z)}{\Phi+\Psi r^{n+1}\log r}$. Since $n\geq 2$ and $\Phi|_{\partial D}\neq 0$, we have $f\in C^2(B(p, \delta)\cap \overline M)$. By Taylor's expansion,
\begin{equation}\label{4-26-a2}
\log (1+f r^{n+1})=f r^{n+1}+ O(f^2 r^{2(n+1)})~\text{as}~ r\rightarrow 0.
\end{equation}
Thus, $[\log (1+f r^{n+1})]_{\alpha\overline\beta}\rightarrow 0$ as $z\rightarrow B(p, \delta)\cap \partial M$ for  $n\geq 2$.
Then combining (\ref{4-26-a1}) and (\ref{4-26-a2}), one has
\begin{equation*}
\frac{\partial^2}{\partial z_\alpha\partial \overline z_\beta}\log \left(1+\frac{\varphi(z)}{k_D(z, \overline z)}\right)\rightarrow 0.
\end{equation*}
As a consequence,
\begin{equation}\label{2019-09-30-a5}
\frac{G_M(z)}{G_D(z)}\rightarrow 1
\end{equation}
as $z\rightarrow\partial M\cap B(p, \delta)$.
From (\ref{2019-09-30-a2}) we have
\begin{equation}\label{2019-09-30-a4}
\frac{k_M(z, \overline z)}{G_M(z)}=\frac{k_D(z, \overline z)}{G_M(z)}+\frac{\varphi(z)}{G_M(z)}.
\end{equation}
Combining (\ref{2019-09-30-a5}) and (\ref{2019-09-30-a4}) we have
\begin{equation}\label{2019-09-30-a6}
\left|\frac{k_M(z, \overline z)}{G_M(z)}-\frac{k_D(z, \overline z)}{G_D(z)}\right|\rightarrow 0
\end{equation}
as $z\rightarrow \partial M\cap B(p, \delta)$. By \cite[Theorem
2]{Di70},  we have
\begin{equation}\label{2019-09-30-a7}
\frac{G_D(z)}{k_D(z, \overline z)}\rightarrow \frac{(n+1)^n\pi^n}{n!}
\end{equation}
 as $z\rightarrow \partial D$.  Substituting (\ref{2019-09-30-a7}) into (\ref{2019-09-30-a6})
  we conclude  the proof of the lemma.
\end{proof}

The following proposition is a  generalization of a result of Fu-Wong \cite[Proposition 1.1]{FW97} which gives a characterization when the Bergman metric on $M\setminus E$ is K\"ahler-Einstein.
 \begin{proposition}\label{2019-09-30-p1}
    Let $M$ be a relatively compact strongly pseudoconvex
    complex manifold with smooth boundary. The Bergman metric on $M\setminus E$ is Kahler-Einstein
     if and only if $B_M(z)=\frac{(n+1)^n\pi^n}{n!}$ for all $z\in M\sm E$.
\end{proposition}
\begin{proof}
If the Bergman metric on $M\setminus E$ is K\"ahler-Einstein, then $R^M_{i\overline j}=cg^M_{i\overline j}$
  where $c$ is a  constant. By Lemma
  \ref{localization of Bergman kernel} and a direct calculation one has that $R^M_{i\overline j}+g^M_{i\overline j}$ goes to zero as a tensor with respect to $\omega^B_M$ when $z\rightarrow \partial M$.  Thus, combining the K\"ahler-Einstein assumption one has $c=-1$ and  this implies
      that $\log B_M(z)$ is a pluriharmonic function on $M\sm E$. Now,
      for any holomorphic disk $\phi: \Delta\ra M\sm E$ with $\phi$ is
      holomorphic in $\Delta:=\{t\in {\mathbb C}: |t|<1\}$,
      smooth continuous up to $\overline {\Delta}$ and $\phi(\partial\Delta)\subset\partial M$, we have $\log
      B_M(\phi(t))$ is harmonic. Since it takes the constant value on the
      boundary by Lemma \ref{2019-09-30-lem1}, it takes a constant value $\log\frac{(n+1)^n\pi^n}{n!}$ over $\Delta$.
      Now, since
    $\partial M$ is strongly pseudoconvex, the union of such disks
    fills up an open subset of $M\sm E$. Since $\log
      B_M$ is real analytic, we conclude that $B_M\equiv
      \log\frac{(n+1)^n\pi^n}{n!}$ over $M\sm E$.
If $\log B_M(z)$ takes constant value, then the Bergman metric is
obviously K\"ahler-Einstien.
\end{proof}

Let $D=\{r>0\}$ be a strongly pseudoconvex domain given in (\ref{localization of Bergman kernel}) where $r$ is a defining for $D$. Then $k_D$ has following expansion
\begin{equation}\label{Bergman kernel asymptotic expansion}
k_D(z, \overline z)=\frac{\Phi(z)}{r^{n+1}(z)}+\Psi(z)\log {r(z)}, z\in D
\end{equation}
with $\Phi, \Psi\in C^\infty(\overline D)$.
Then from Proposition \ref{2019-09-30-p1}  we have the following
\begin{lemma}\label{2019-09-30-a11}
Let $M$ be a relatively compact strongly pseudoconvex complex manifold with smooth boundary. Assume the Bergman metric on $M\setminus E$ is Kahler-Einstein. Then
\begin{equation}
\Psi(z)=O(r^k)~\text{on}~ D\cap B(p, \delta)
\end{equation}
for any $k>0$.
\end{lemma}
\begin{proof}
By Proposition \ref{2019-09-30-p1} we have  the same identities as in \cite[(1.1)]{FW97}. Thus,
\begin{equation}\label{2019-09-30-a8}
J(k_M)=(-1)^nC_n k_M^{n+2}~\text{on}~D\cap B(p, \delta),
\end{equation} where $C_n=\frac{(n+1)^n\pi^n}{n!}$. On the other hand, \begin{equation}\label{2019-09-30-a9}
k_M=k_D+\varphi(z)
\end{equation} when $z\in B(p, \delta)\cap D$,
 where $\varphi\in C^\infty(B(p, \delta)\cap\overline D)$. Substituting (\ref{Bergman kernel asymptotic expansion}) and (\ref{2019-09-30-a9})
 into (\ref{2019-09-30-a8}) and by a similar argument as in the proof of
 \cite[Theorem 2.1]{FW97} we get the conclusion of the lemma.

\end{proof}
Let $\Omega\subset \mathbb C^n$ be a bounded strongly pseudocovnex domain with smooth boudnary.
The following Monge-Ampere type equation on $\Omega$ was introduced by Fefferman \cite{Fe76}
\begin{equation}
\begin{split}
J(u)\equiv(-1)^n\det{(\begin{array}{cc}
    u& u_{\overline \beta}\\
    u_{\alpha}& u_{\alpha\overline \beta}
    \end{array})}&=1~\text{in}~\Omega\\
u&=0~\text{at} ~\partial \Omega
\end{split}
\end{equation}
Fefferman proved that $\Omega$ has smooth defining function $r_F$ which satisfies $$J(r_F)=1+O(r_F^{n+1}).$$ We call $r_F$ Fefferman's defining function for $\Omega$. Let us recall Fefferman's construction of such defining function. The existence of such $r_F$ can be  established in the following steps: Starting with $\Omega=\{r>0\}$ and $dr|_{\partial \Omega}\neq 0$, Fefferman defined recursively
\begin{equation}\label{Fefferman's construction}
\begin{split}
u^1&=\frac{r}{(J(r))^{1/n+1}},\\
u^s&=u^{s-1}\left(1+ \frac{1-J(u^{s-1})}{[n+2-s]s}\right), 2\leq s\leq n+1.
\end{split}
\end{equation}
Each $u^s$  satisfies $J(u^s)=1+O(r^s)$ and $u^{n+1}$ is what we
call Fefferman defining function.

\begin{lemma}\label{10-4-lem1}
    There exists a Fefferman's defining function $r_F$ for $D$ such that
    \begin{equation}\label{2019-09-30-a13}
     r_F=\left(\frac{\pi^n}{n!}k_M\right)^{-\frac{1}{n+1}}~\text{on}~D\cap B(p, \sigma).
    \end{equation}
    for some small $\sigma$.
\end{lemma}
\begin{proof}
First, by  Lemma \ref{localization of Bergman kernel} we have $k_M=k_D+\varphi(z)$. Then  from the Bergman kernel expansion of $k_D$ we have
\begin{equation}
\begin{split}
k_M(z, \overline z)&=k_D+\varphi=\frac{\Phi(z)}{r^{n+1}}+\Psi(z)\log r+\varphi\\
&=\frac{\Phi+r^{n+1}\Psi\log r+r^{n+1}\varphi}{r^{n+1}}
\end{split}
\end{equation}
when $z\in D\cap B(p, \delta)$. Since $k_M(z, \overline z)>0$ one has
$$\Phi+r^{n+1}\Psi\log r+r^{n+1}\varphi>0$$ for all $z\in D\cap B(p, \delta)$. Thus,
\begin{equation}
(k_M)^{-\frac{1}{n+1}}(z)=\frac{r}{(\Phi+r^{n+1}\Psi\log r+r^{n+1}\varphi)^{\frac{1}{n+1}}}
\end{equation}
is well-defined on $D\cap B(p, \delta)$. Moreover, from Lemma
\ref{2019-09-30-a11} we have that $(k_M)^{-\frac{1}{n+1}}\in
C^\infty(B(p, \delta)\cap \overline D)$. Then by partition of unity,
we can choose a defining funciton $r_0$ for $D$ such that
\begin{equation}\label{10-8-a1}
r_0=(\frac{\pi^n}{n!}k_M)^{-\frac{1}{n+1}}~\text{on}~D\cap B(p, \frac{\delta}{2}).
\end{equation}
This idea  has been crucially used in Huang-Xiao \cite{HX16} to
construct a Fefferman's defining function which satisfy the
Monge-Ampere equation.

Let $r_F$ be a Fefferman defining function for $D$. Then $r_F=h r_0$ for some $h\in C^\infty(\overline D)$ and $h>0$ on $D$. Since $$J(r_F)=h^{n+1}J(r_0)~\text{on}~\partial D$$
and $J(r_F)=1$ on $\partial D$, thus $J(r_0)\neq 0$ on $\partial D$. Thus, by continuty $J(r_0)\neq 0$ in a neighborhood of $\partial D$. So the set $K=\{z\in D: J(r_0)=0\}$ is a compact subset of $D$. Choose a cut-off function $\chi$ such that $\chi\equiv 1$ in a neighborhood of $\partial D$ and $\chi\equiv 0$ in a neighborood of $K$. Set $$u^1=\chi \frac{r_0}{(J(r_0))^{\frac{1}{n+1}}}.$$ Then we still have $J(u^1)=1$ on $\partial D$.
We notice that the Kahler-Einstein condition of the Bergman metric implies that
$J(\frac{\pi^n}{n!}k_M)^{-\frac{1}{n+1}}=1$ for $z\in D$, so $J(r_0)\equiv 1$ on $D\cap B(p, \frac{\delta}{2})$ by the construction of $r_0$ in (\ref{10-8-a1}). Then
\begin{equation}\label{2019-09-30-a12}
J(u^1)=1~\text{on}~D\cap B(p, \sigma),
\end{equation}
for some $\sigma<\frac{\delta}{2}.$
Then from Fefferman's construction of Fefferman defining function (\ref{Fefferman's construction})  we see that
\begin{equation}
u^1=u^2=\cdots=u^{n+1}=r_0~\text{on}~D\cap B(p, \sigma).
\end{equation}
Combing with (\ref{2019-09-30-a12}) and changing the values of $u_{n+1}$
in a certain compact subset of $M$ if needed,  we get the conclusion
of the lemma.
 \end{proof}
\section{Proof of Theorem \ref{mthm1}}
\begin{proof}
For any $p\in\partial M$, let $D$ and $B(p, \delta)$ be the sets as chosen in lemma \ref{localization of Bergman kernel}. Let $r_F$ be the Fefferman defining for $D$ function as chosen in lemma \ref{10-4-lem1}.
By Fefferman's Bergman asymptotic expansion on $D$, we have
\begin{equation}
k_D(z, z)=\frac{\phi}{r_F^{n+1}}+\psi \log r_F,
\end{equation}
where $\phi, \psi\in C^\infty(\overline D)$ and $\phi|_{\partial D}\neq 0$.
On the other hand, by lemma \ref{localization of Bergman kernel},
$$k_M(z, \overline z)=k_D(z, \overline z)+\varphi(z), z\in B(p, \delta)\cap D$$ where $\varphi\in C^\infty(B(p, \delta)\cap \overline D)$. Thus,
\begin{equation}\label{2019-09-30-a14}
k_M r_F^{n+1}=\phi+\psi r_F^{n+1}\log r_F+\varphi r_F^{n+1} ~\text{on}~B(p, \delta)\cap D.
\end{equation}
Substituting (\ref{2019-09-30-a13}) to (\ref{2019-09-30-a14}) we have
\begin{equation}\label{4-19-a1}
\frac{n!}{\pi^n}=\phi+\psi r_F^{n+1}\log r_F+\varphi r_F^{n+1}~\text{on}~D\cap B(p, \sigma).
\end{equation}
By \cite[Lemma 2.2]{FW97}, we have
\begin{equation}\label{4-19-a2}
\phi-\varphi r_F^{n+1}-\frac{n!}{\pi^n}=O(r_F^k), \psi=O(r_F^k) ~\text{on}~D\cap B(p, \sigma), \forall k>0.
\end{equation}
Thus,
\begin{equation}\label{4-19-a3}
\phi-\frac{n!}{\pi^n}=O(r_F^{n+1})~\text{on}~D\cap B(p, \sigma).
\end{equation}
When $n=2$, $\psi=O(r_F^k) ~\text{on}~D\cap B(p, \sigma), \forall k>0$ implies that $\partial D\cap B(p, \sigma)$ is spherical by a result of Burns-Graham \cite[pp.129]{Gr85} (also see \cite[pp.23]{BdM90}).
When $n\geq 3$, it follows from (\ref{4-19-a3}) that $\partial D\cap B(p, \sigma)$ is spherical by combining the Morse normal form theory \cite{CM74} and a result of Christoffers \cite{Ch81}
 as argued in the work of Huang-Xiao \cite[pp.6-7]{HX16}.
Thus, we get the conclusion of  Theorem \ref{mthm1}.
\end{proof}

Theorem \ref{mthm-4-18} is a direct corollary of Theorem
\ref{mthm1}. Huang \cite{H06} proved  that  a Stein space with
possible isolated normal singularities and compact strongly
pseudoconvex and algebraic boundary is biholomorphic to a ball
quotient. Then a direct corollary of  Theorem \ref{mthm-4-18} and
\cite[Theorem 3.1]{H06} is the following
\begin{corollary}
Let $\Omega$ be a stein space with isolated  normal  singularities
and a compact smooth boundary $\partial \Omega$. Assume  the
$\partial \Omega$ is CR equivalent to an algebraic CR manifold in a
complex Euclidean space. If the Bergman metric $\omega_\Omega^B$ on
$\hbox{Reg}(\Omega)$ is Kahler-Einstein then $\Omega$ is
biholomorphic to a ball quotient $\mathbb {B}^n/\Gamma$ where
$\Gamma\subset {\rm Aut}({\mathbb B}^n)$ is finite subgroup with
$0\in {\mathbb B}^n$ the only fixed point of any non-identity
element of $\Gamma$.
\end{corollary}

\section{Bergman metric on a ball quotient}
Let $\Omega:=\mathbb B^n/\Gamma$ where $\Gamma$ is a finite subgroup
of ${\rm Aut(\mathbb B^n)}$ with $0$ as the  unique fixed point for
each non-identity element.
Then $\Omega$ is a stein space with only an isolated singularity.
Let $\pi: \mathbb B^n\rightarrow \mathbb B^n/\Gamma$ be the standard
branched covering map. Write $p=\pi(0)$.  Let $\omega^B$ be the
Bergman metric on $\Omega$. Let $A^2(\Omega)$ be the
$L^2$-integrable holomorphic $(n, 0)$-forms on ${\rm Reg}(\Omega)$.
Let $\{\alpha_j\}_{j=1}^\infty$ be an orthnormal basis of
$A^2(\Omega)$. Locally, write $\alpha_j=a_j dw, j\geq 1$ and
$k_{\Omega}(w, \overline w)=\sum_{j=1}^\infty|a_j|^2$. Then
$\omega_\Omega^B=i\partial\overline\partial \log k_{\Omega}(w,
\overline w)$. Write $\pi^\ast \alpha_j=f_j dz$ where
$dz=dz_1\wedge\cdots\wedge dz_n$ and $\{f_j\}$ are holomorphic
functions on $\mathbb B^n\setminus\{0\}$. By the Hartogs extension
theorem, $\{f_j\}$ can be holomorphically extended to $\mathbb B^n$.
Moreover, $f_j$ satisfies $$f_j\circ \gamma (z)\det\gamma=f_j(z),
\forall \gamma\in\Gamma, \forall z\in\mathbb B^n.$$

Set $A^2_{\Gamma}(\mathbb B^n)=\{f\in A^2(\mathbb B^n): f\circ \gamma \det\gamma=f,
\forall \gamma\in \Gamma\}.$ Then $A^2_{\Gamma}(\mathbb B^n)$ is a closed
 subspace of $A^2(\mathbb B^n)$. Let $P_\Gamma: L^2(\mathbb B^n)\rightarrow A^2_{\Gamma}
 (\mathbb B^n)$ be the orthogonal projection. Let $\{f_j\}_{j=1}^\infty$ be an
  orthnormal basis of $A^2_{\Gamma}(\mathbb B^n)$. Write $$K_{\Gamma}(z, \overline w)=
  \sum_{j=1}^\infty f_j(z)\overline f_j(w), z, w\in \mathbb B^n.$$ $K_\Gamma(z, \overline w)$
 is then the Schwarz kernel of $P_\Gamma$. That is,
$$P_\Gamma f=\int_{\mathbb B^n} K_{\Gamma}(z, \overline w) f(w)dv$$ where
 $dv$ is the Lebesgue measure on $\mathbb C^n$.
Define $$Q_\Gamma f=\int_{\mathbb B^n} \frac{1}{|\Gamma|}\sum_{\gamma\in\Gamma}
 K(\gamma z, \overline w) \det\gamma f(w)dv, \forall f\in L^2(\mathbb B^n)$$
where $K(z, \overline w)$ is the Bergman kernel function of the
$\mathbb B^n$.
 Then $K(z, \overline w)=\frac{n!}{\pi^n}\frac{1}{(1-z\cdot \overline w)^{n+1}}$ and
  $z\cdot\overline w=z_1\overline w_1+\cdots+z_n\overline w_n$. Then
$Q_\Gamma f\in A^2_\Gamma(\mathbb B^n)$ for all $f\in L^2(\mathbb B^n)$. Moreover,
\begin{equation}\label{19-11-16a}
\frac{1}{|\Gamma|}\sum_{\gamma\in\Gamma} K(\gamma z, \overline {\tau w})
\det\gamma\det\overline\tau=\frac{1}{|\Gamma|}\sum_{\gamma\in\Gamma} K(\gamma z,
\overline { w})\det\gamma, \forall \tau \in\Gamma.
\end{equation}
In fact, $\overline{\tau^t}=\tau^{-1}\in\Gamma, \forall \tau\in\Gamma$ where $\tau^t$ is the transpose matrix of $\tau$, then
\begin{equation*}
\begin{split}
\frac{1}{|\Gamma|}\sum_{\gamma\in\Gamma}K(\gamma z, \overline{\tau w} )\det\gamma \det{\overline\tau}
&=\frac{c_n}{|\Gamma|}\sum_{\gamma\in \Gamma} \frac{1}{(1-z^t \gamma^t\cdot \overline {\tau }\overline w)^{n+1}}\det\gamma \det{\overline\tau}\\
&=\frac{c_n}{|\Gamma|}\sum_{\gamma\in\Gamma}\frac{1}{(1-z^t(\overline\tau^t\gamma)^t\overline w)^{n+1}}\det({\overline \tau^t}\gamma )\\
&=\frac{1}{|\Gamma|}\sum_{\gamma\in\Gamma} K(\gamma z, \overline w)\det\gamma.
\end{split}
\end{equation*}
Here, $c_n=\frac{n!}{\pi^n}$.
\begin{lemma}\label{19-11-lem1}
\begin{equation}
Q_\Gamma=P_\Gamma~\text{on}~L^2(\mathbb B^n); ~~
K_{\Gamma}(z, \overline w)=\frac{1}{|\Gamma|}\sum_{\gamma\in \Gamma}K(\gamma z, \overline w)\det\gamma.
    \end{equation}
\begin{proof}
    For all $f\in L^2(\mathbb B^n)$, write $f=f_1+f_2$ where $f_1=P_\Gamma f$ and $f_1\perp f_2$ and $f_2\perp A^2_{\Gamma}(\mathbb B^n)$.
    By (\ref{19-11-16a}), one has
    \begin{equation}
    \begin{split}
    Q_\Gamma f
    &=\frac{1}{|\Gamma|}\int_{\mathbb B^n}\sum_{\gamma\in\Gamma} K(\gamma z, \overline w)\det\gamma f_1(w)dv+\frac{1}{|\Gamma|}\int_{\mathbb B^n}\sum_{\gamma\in\Gamma} K(\gamma z, \overline w)\det\gamma f_2(w)dv\\
    &=\frac{1}{|\Gamma|}\int_{\mathbb B^n}\sum_{\gamma\in\Gamma} K(\gamma z, \overline w)\det\gamma f_1(w)dv\\
    &=\frac{1}{|\Gamma|}\sum_{\gamma\in\Gamma}\det\gamma f_1(\gamma z)=f_1(z)\\
    &=P_{\Gamma}f.
    \end{split}
    \end{equation}
    As a consequence, $Q_\Gamma$ and $P_\Gamma$ have the same Schwarz kernel. Thus, we get the conclusion of the second part of the lemma.
\end{proof}
\end{lemma}
Write $\omega_\Gamma=i\partial\overline\partial \log K_\Gamma(z, \overline z)$. Then we have the following
\begin{lemma}
\begin{equation}
\pi^\ast \omega_\Omega^B=\omega_\Gamma.
\end{equation}
Moreover,   $\omega_\Omega^B$ is K\"ahler-Einstein if and only if
$\omega_\Gamma$ is K\"ahler-Einstein on $\mathbb B^n\setminus\{0\}$.
\end{lemma}
\begin{proof}
    Let $\{\alpha_j\}$ be an orthnormal basis of $A^2(\Omega)$. Write $\alpha_j=a_jdw$ and $\pi^\ast \alpha_j=f_j dz$ on $\mathbb B^n\setminus{\{0\}}$. Here  $w=(w_1, \cdots, w_n)$ are local coordinates on ${\rm Reg}~\Omega$ and $dw=dw_1\wedge\cdots\wedge dw_n$.
    We have $a_j\circ\pi \det \pi^{'}=f_j$. Since $\Gamma\subset{\rm Aut}~(\mathbb B^n)$, then $|\det{\pi'}|^2=1$. Thus,
    \begin{equation}\label{19-11-19-a1}
    |a_j\circ \pi|^2=|f_j|^2, \forall j.
    \end{equation}
    \begin{equation}
    \frac{1}{i^{n^2}}\int_{\mathbb B^n}f_j\overline {f_k}dz\wedge d\overline z=\frac{1}{i^{n^2}}\int_{\mathbb B^n}\pi^\ast\alpha_j\wedge\overline{\pi^\ast\alpha_k}=\frac{1}{i^{n^2}}|\Gamma|\int_{\Omega}\alpha_j\wedge\overline{\alpha_k}=|\Gamma|\delta_{jk}.
    \end{equation}
    For any $f\in A^2_{\Gamma}(\mathbb B^n)$, there exist an $\alpha\in A^2(\Omega)$ such that $\pi^\ast\alpha=f(z)dz$.
    Thus, $\{\frac{1}{\sqrt{|\Gamma|}} f_j\}$ is an orthonormal basis of $A^2_{\Gamma}(\mathbb B^n)$.  Then combine with (\ref{19-11-19-a1})
    \begin{equation}
    K_{\Gamma}(z, \overline z)=\frac{1}{|\Gamma|}\sum_{j=1}^\infty|f_j(z)|^2=\frac{1}{|\Gamma|}|a_j\circ \pi|^2=\frac{1}{|\Gamma|}\pi^\ast k_{\Omega}.
    \end{equation}
    By taking the $\partial\overline\partial\log$ on both sides of the above equation we get the conclusion of the lemma.
    \end{proof}

Assume that $\omega_\Omega^B$ is K\"ahler-Einstein. Then
$\omega_{\Gamma}$ is Kahler-Einstein on $\mathbb B^n\setminus\{0\}$.
The Bergman kerenl on $\mathbb B^n$ is denoted by $K(z, \overline
z)$. Then
$$K(z, \overline z)=\frac{n!}{\pi^n}\frac{1}{(1-|z|^2)^{n+1}}.$$ By Lemma \ref{19-11-lem1}
\begin{equation}
\begin{split}
K_{\Gamma}(z, \overline z)&=\frac{1}{|\Gamma|}\sum_{\gamma\in\Gamma}K(\gamma z, \overline z)\det\gamma=\frac{1}{|\Gamma|}\frac{n!}{\pi^n}\sum_{\gamma\in\Gamma}\frac{1}{(1-\gamma z\cdot\overline z)^{n+1}}\det\gamma\\
&=\frac{n!}{\pi^n}\frac{1}{|\Gamma|}\left[\frac{1}{(1-|z|^2)^{n+1}}+\Psi(z)\right],
\end{split}
\end{equation}
where $\Psi=\sum_{\gamma\neq id}\frac{1}{(1-\gamma z\cdot\overline z)^{n+1}}\det\gamma.$ Since $1-\gamma z\cdot\overline z\neq 0, \forall z\in \partial B^n$ when $\gamma\neq id$, it follows that $\Psi(z)\in C^\infty(\overline {\mathbb B^n})$.
Then
\begin{equation}\label{19-11-19-a2}
\omega_\Gamma=i\partial\overline\partial\log K_\Gamma=i\partial\overline\partial\log{\frac{1}{(1-|z|^2)^{n+1}}}+i\partial\overline\partial\log{(1+\tilde\Psi)}
\end{equation}
where $\tilde\Psi=\Psi(z)(1-|z|^2)^{n+1}.$ Write $\omega_\Gamma=i \sum_{i, j=1}^ng_{i\overline j}dz_i\wedge d\overline z_j$. By direct calculation,
\begin{equation}
g_{i\overline j}=(n+1)\left\{\frac{\delta_{ij}}{1-|z|^2}+\frac{\overline z_i z_j}{(1-|z|^2)^2}\right\}+O((1-|z|^2)^{n-1}).
\end{equation}
Here, we mention that  $O(f)$ means that there exist a constant $C>0$ such that the term can be bounded by $C|f|$ near $\partial B^n$.
Then
\begin{equation}\label{19-11-19-a4}
\begin{split}
\det{g_{i\overline j}}=(n+1)^n\frac{1}{(1-|z|^2)^{n+1}}+O((1-|z|^2)^{-n+1})\\
=(n+1)^n\frac{1}{(1-|z|^2)^{n+1}}[1+O((1-|z|^2)^2)].
\end{split}
\end{equation}
Then the Ricci curvature with respect to $\omega_\Gamma$ is given by
\begin{equation}\label{19-11-19-a3}
\begin{split}
\Theta_{\Gamma}&=i\overline\partial\partial\log{\det g_{i\overline j}}=-(n+1)i\overline\partial\partial\log{(1-|z|^2)}+\overline\partial\partial [O((1-|z|^2)^2)]\\
&=-(n+1)i\overline\partial\partial\log{(1-|z|^2)}+O(1).
\end{split}
\end{equation}
Since $\omega_\Gamma$ is Kahler-Einstein on $\mathbb B^n\setminus\{0\}$, then $\Theta_\Gamma=c_0\omega_\Gamma$ where $c_0$ is a constant. From (\ref{19-11-19-a2}) and (\ref{19-11-19-a3}) we have
\begin{equation}
-(n+1)\overline\partial\partial\log{(1-|z|^2)}+O(1)=c_0[-(n+1)\partial\overline\partial\log{(1-|z|^2)}+\partial\overline\partial \log{(1+\tilde\Psi)}].
\end{equation}
Letting $z\rightarrow \partial\mathbb B^n$, we have $c_0=-1$.
\begin{theorem}\label{4-24-l1}
Set $u=\log{K_{\Gamma}}$. Then the Bergman metric on
${\rm Reg}(\Omega)$ is K\"ahler-Einstein  with $n\ge 2$ if  $u$
satisfies the following complex Monge-Ampere equation
\begin{equation}\label{MAequ}
\det (u_{i\overline j})=c e^u ~\text{on}~\mathbb B^n\setminus\{0\},
\ u|_{\partial {\mathbb B}_n}=\infty.
\end{equation}
where $c= \frac{(n+1)^n \pi^n |\Gamma|}{n!}.$ Conversely, if $u$
satisfies (\ref{MAequ}), then the Bergman metric on
${\rm Reg}(\Omega)$ is K\"ahler-Einstein.
\end{theorem}

\begin{proof} We only need to prove the necessary part.
 The proof is similar to that for $B_M=const$.
From $\Theta_\Gamma=-\omega_\Gamma$, we have  that $\log({\det u_{i\overline j}})-u$ is a pluriharmonic function on $\mathbb B^n\setminus\{0\}$. Write $v=\log({\det u_{i\overline j}})-u$.  Since $n\geq 2$, then $v$ can be smoothly extended to $\mathbb B^n$ which is still denoted by $v$. Then $v$ is a pluriharmonic function on $\mathbb B^n$. Thus, $u=\log{K_{\Gamma}}$ satisfies the following
\begin{equation}\label{19-11-19-a5}
\det{u_{i\overline j}}=e^ve^u.
\end{equation}
Substituting (\ref{19-11-19-a4}) to (\ref{19-11-19-a5}) we have
\begin{equation}
\frac{(n+1)^n}{(1-|z|^2)^{n+1}}[1+O((1-|z|^2)^2)]=e^v \frac{n!}{\pi^n}\frac{1}{|\Gamma|}\left[\frac{1}{(1-|z|^2)^{n+1}}+\Psi(z)\right].
\end{equation}
Letting $z\rightarrow \partial\mathbb B^n$, we have
$$e^v\rightarrow \frac{(n+1)^n \pi^n |\Gamma|}{n!}.$$
Since $v$ is pluriharmonic on $\mathbb B^n$, then
$$e^v\equiv   \frac{(n+1)^n \pi^n |\Gamma|}{n!},  \forall z\in\mathbb B^n.$$
Thus, $u=\log{K_{\Gamma}}$ satisfies the following Monge-Ampere equation
\begin{equation} \label{1010}
\det u_{i\overline j}=c e^u,
\end{equation}
where $c= \frac{(n+1)^n \pi^n |\Gamma|}{n!}.$
\end{proof}

We notice that if $u=\log K_\Gamma$ satisfies (\ref{1010}) with
$K_\Gamma(0,0)\not = 0$, then by continuity $\omega_\Gamma$ is a
well-defined complete Ka\"hler-Einstein metric over ${\mathbb B}^n$.
Hence, by the uniqueness of the Cheng-Yau metric \cite{CY80},
$\omega_\Gamma$ is a hyperbolic metric and thus by the
uniformization theorem, we see that $\Gamma=\{\hbox{id}\}$ and thus
$\Omega$ is biholomorphic to the ball. Namely, we have the
following:
\begin{corollary}
Let $\Gamma\subset \hbox{Aut}_0({\mathbb B}^n)$ with $n\ge 2$ be a
non-trivial finite subgroup with $0$ as the only fixed point for
each non-identity element of $\Gamma$. Let $K_{\Gamma}$ be the
function defined in (\ref{19-11-lem1}). If $K_{\Gamma}(0,0)\not =0$,
then the Bergman metric of ${\rm Reg}({\mathbb B}^n/\Gamma)$ is not K\"ahler-Einstein.
\end{corollary}
%


\begin{example}
Suppose $\Omega=\mathbb B^3/\Gamma$, where $\Gamma=\{\gamma_1, \gamma_2\}$ and $\gamma_1=id, \gamma_2={\rm diag}(-1, -1, -1)$.
\begin{equation}
K_{\Gamma}=\frac{3}{\pi^3}\left[ \frac{1}{(1-|z|^2)^4}-\frac{1}{(1+|z|^2)^4} \right]=\frac{4!}{\pi^3}\frac{|z|^2(1+|z|^4)}{(1-|z|^4)^4}.
\end{equation}
Thus, $$K_{\Gamma}(0, 0)=0.$$
Set $u=\log K_{\Gamma}$.
Then \begin{equation}
\begin{split}
u&=\log {\frac{4!}{\pi^3}}+\log{|z|^2}+\log{(1+|z|^4)}-\log{(1-|z|^4)^4}\\
&=\log {\frac{4!}{\pi^3}}+\log |z|^2+5|z|^4+O(|z|^8).
\end{split}
\end{equation}
By direct calculation,
\begin{equation}
\begin{split}
u_{1\overline 1}&=\frac{|z_2|^2}{|z|^4}+10|z|^2+10|z_1|^2+O(|z|^6), ~u_{1\overline 2}=-\frac{1}{|z|^4}\overline z_1 z_2+10\overline z_1 z_2+O(|z|^6)\\
u_{2\overline 1}&=-\frac{1}{|z|^4}z_1\overline z_2+10z_1\overline z_2+O(|z|^6),~~~~~ u_{2\overline 2}=\frac{|z_1|^2}{|z|^4}+10|z|^2+10|z_2|^2+O(|z|^6).
\end{split}
\end{equation}
Then $\det u_{i\overline j}(0)=20$, but $K_{\Gamma}(0, 0)=0$. Thus,  it follows that $u=\log {K_\Gamma}$ does not satisfy the Monge-Ampere equation (\ref{MAequ}).  Hence, the Bergman metric on $\Omega$ is not K\"ahler-Einstein.
\end{example}

When $n=1$ and for any finite subgroup $\Gamma\subset{\rm
Aut}(\mathbb B^1)$,  assume $|\Gamma|=r, 1\leq r<\infty.$ It is well
known that $\Gamma=\{1, e^{2\pi i\frac{1}{r}}, \cdots, e^{2\pi
i\frac{r-1}{r}}\}.$ Thus, on $\mathbb B^1$
\begin{equation}
\begin{split}
K_{\Gamma}(z, \overline z)
&=\frac{1}{\pi|\Gamma|}\sum_{\gamma\in\Gamma}\frac{1}{(1-\gamma z\cdot\overline z)^{2}}\det\gamma=\frac{1}{\pi r}\sum_{j=1}^r\frac{1}{[1-e^{2\pi i\frac{j}{r}}|z|^2]^2}e^{2\pi i \frac{j}{r}}.
\end{split}
\end{equation}
By Taylor's expansion,
\begin{equation}
K_{\Gamma}=\frac {1}{\pi}\sum_{j=1}^r\sum_{k=0}^\infty (k+1) e^{2\pi i\frac{j}{r}(k+1)}|z|^{2k}=\frac{r}{\pi}\sum_{k=1}^\infty k|z|^{2(kr-1)}=\frac{r}{\pi}\frac{|z|^{2(r-1)}}{(1-|z|^{2r})^2}.
\end{equation}
Set $u=\log K_\Gamma.$ Then $u_{1\overline 1}=2r^2\frac{|z|^{2(r-1)}}{(1-|z|^{2r})^2}$.
 Since $c=2\pi r$, then one sees immediately that
\begin{equation}
u_{1\overline 1}=c e^u~\text{on}~\mathbb B^1\setminus\{0\}.
\end{equation} Notice that the sufficient part of Theorem
\ref{4-24-l1} holds even for $n=1$.  We have the following:
\begin{proposition}
For any finite subgroup $\Gamma\subset{\rm Aut}_0(\mathbb B^1)$, its
Bergman metric on ${\rm Reg}(\mathbb B^1/\Gamma)$ is
K\"ahler-Einstein.
\end{proposition}

We finish off  this paper by recalling the following generalized
Cheng conjecture formulated in \cite{HX20}:
\begin{conjecture}
Let $\Omega$ be a normal Stein space  with a compact spherical boundary of
complex dimension $n\ge 2$. If the Bergman metric over ${\rm
Reg}(\Omega)$ is K\"ahler-Einstein, then $\Omega$ is biholomorphic
to ${\mathbb B}^n$.
\end{conjecture}

\begin{center}
{\bf Acknowledgement}
\end{center}

The second author would like to express his gratitude to the Department of Mathematics of Rutgers University for hospitality and providing excellent working conditions during his visit from August 2019 to July 2020.


\noindent Xiaojun Huang (huangx@math.rutgers.edu),
Department of Mathematics, Rutgers University, New Brunswick, NJ
08903, USA.

\medskip
\noindent Xiaoshan Li (xiaoshanli@whu.edu.cn), School of Mathematics
and Statistics, Wuhan University, Hubei 430072,  China.
\end{document}